\documentclass[reqno,onecolumn,oneside]{paper}
\usepackage[english]{babel}
\usepackage{enumerate,amsmath,amsthm,amsfonts,%ulem,%pifont,
amssymb,latexsym,mathrsfs,%graphicx,graphics
}
\usepackage[normalem]{ulem}
\usepackage[all]{xy}
\usepackage[sort]{cite}

\usepackage[mathscr]{euscript}
\let\euscr\mathscr \let\mathscr\relax
\usepackage[scr]{rsfso}
\usepackage{tikz-cd}

\newtheorem{theorem}{Theorem}[section]
\newtheorem{proposition}[theorem]{Proposition}
\newtheorem{lemma}[theorem]{Lemma}
\newtheorem{corollary}[theorem]{Corollary}
\theoremstyle{definition}
\newtheorem{definition}[theorem]{Definition}

\theoremstyle{remark}
\newtheorem{remark}[theorem]{Remark}

\renewcommand{\wp}{\euscr{P}}

\newcommand{\bydef}{\mathrel{\mathop:}=}

\newcommand{\ax}{\operatorname{Ax}}

\newcommand{\id}{\operatorname{id}}

\newcommand{\rMod}{\mathcal{M}\!\!\:\mathit{od}\textrm{-}}
\newcommand{\biMod}{\textrm{-}\mathcal{M}\!\!\:\mathit{od}\textrm{-}}
\newcommand{\Mod}{\textrm{-}\mathcal{M}\!\!\:\mathit{od}}

\newcommand{\lang}{\mathcal{L}}
\newcommand{\fml}{\mathit{Fm}_\mathcal{L}}
\newcommand{\Fml}{\wp\fml}
\newcommand{\fmll}{\mathit{Fm}_{\mathcal{L}_1}}
\newcommand{\Fmll}{\wp\fmll}
\newcommand{\fmlll}{\mathit{Fm}_{\mathcal{L}_2}}

\newcommand{\fmlu}{\mathit{Fm}_{\mathcal{L}_1 \cup \mathcal{L}_2}}

\newcommand{\sfm}{\Sigma_{\cat L}}
\newcommand{\Sfm}{\wp\sfm}
\newcommand{\sfmm}{\Sigma_{\cat{L}_1}}
\newcommand{\Sfmm}{\wp\sfmm}
\newcommand{\sfmmm}{\Sigma_{\cat{L}_2}}
\newcommand{\Sfmmm}{\wp\sfmmm}
\newcommand{\sfmu}{\Sigma_{\cat{L}_1 \cup \cat{L}_2}}
\newcommand{\Sfmu}{\wp\sfmu}

\newcommand{\SL}{\mathcal{SL}}
\newcommand{\Q}{\mathcal{Q}}

\newcommand{\cat}{\mathcal}

\newcommand{\Eq}{\mathit{Eq}}
\newcommand{\seq}{\mathit{Seq}}

\renewcommand{\th}{\mathit{Th}}

\newcommand{\Var}{\operatorname{Var}}

\renewcommand{\phi}{\varphi}
\renewcommand{\theta}{\vartheta}
\newcommand{\e}{\varepsilon}
\newcommand{\g}{\gamma}
\renewcommand{\d}{\delta}

\newcommand{\restr}{\upharpoonright}

\newcommand{\under}{\backslash}

\newcommand{\lto}{\longrightarrow}
\newcommand{\To}{\Rightarrow}
\newcommand{\lmapsto}{\longmapsto}

\newcommand{\ov}{\overline}

\newcommand{\tensor}{\otimes}

\def\amslatex\slash{{\protect\AmS-\protect\LaTeX}}

\begin{document} 

\title{Coproduct and amalgamation of deductive systems by means of ordered algebras}

\author{Ciro Russo} 

\institution{Departamento de Matemática \\ Universidade Federal da Bahia, Brazil \\ \small{\texttt{ciro.russo@ufba.br}}}

%\email{ciro.russo@vanderbilt.edu}

\maketitle
\today

\begin{abstract}
We propose various methods for combining or amalgamating propositional languages and deductive systems. We make heavy use of quantales and quantale modules in the wake of previous works by the present and other authors. We also describe quite extensively the relationships among the algebraic and order-theoretic constructions and the corresponding ones based on a purely logical approach.

\vspace{0.3cm}
{\bf Acknowledgements.} This paper has been awarded the \textbf{2021 Newton da Costa Prize for Logic} and will be presented at the \textbf{2nd World Logic Prizes Contest} within the \textbf{UNILOG 2022} conference, in Crete.

This work was supported by the individual travel grant \emph{Professor Visitante no Exterior S\^enior} - \textbf{Grant No. 88887.477515/2020-00}, awar\-ded by the Coordenadoria de Aperfei\c coamento de Pessoal de N\'ivel Superior and the Universidade Federal da Bahia through the CAPES-PrInt UFBA.
\end{abstract}

\section*{Introduction}

The relationships among different logics and languages constitute a rather important and interesting topic in various areas of both pure and applied Mathematical Logic, such as Abstract Algebraic Logic, Proof Theory, and Automated Deduction. A recent successful approach to such a topic involves the use of ordered algebraic structures and residuation theory \cite{blja,wadi} more as metalogical tools rather than algebraic semantics as it is common in the literature.

The approach to deductive systems by means of order theory has a quite long history, which traces back to Tarski \cite{tar} and goes through a large part of the twentieth century and the last decades -- see, for instance, \cite{bljo,cze,woj}. However, the complementary role of Algebra seems to have been fully understood only recently. Indeed, the representation of consequence relations by means of closure operators seemed to have reached a cul-de-sac due to its inability to manage the syntax of a deductive system. On the contrary, adding an algebraic structure to the abstract representation of the domains of deductive systems and to their lattice of theories, which is in our opinion the main achievement of Galatos and Tsinakis' paper \cite{galtsi}, made the abstract approch flexible enough to fully handle any propositional logic.

Quantales were introduced by Mulvey \cite{mulvey}, suggested by observations related to non-commutative C$^*$-algebras, constructive foundations for
quantum mechanics, and non-commutative logics. Later on, quantales and quantale modules were used in the study of algebraic and logical foundations of quantum mechanics indeed \cite{abrvic,moore,resende1}, but eventually they proved to be rather useful also in other areas of pure and applied mathematics, such as non-commutative topology \cite{conmir,berni,borc1,borc2}, Linear Logic \cite{yetter}, and data compression (see, e.g., \cite{dnr,russojlc}, among others).

However, despite of their multiple applications, only quantales have been systematically studied since their inception \cite{rosenthal}, while the first methodical widenings on quantale modules are rather recent \cite{russothesis,russojlc,russoapal,solo}. 

In the aforementioned paper \cite{galtsi} by Galatos and Tsinakis, the authors proposed an enriched perspective on the theory of consequence operators. Closure operators on a powerset lattice are able to describe the deductive part of a logic, while its structural one is left aside. As a matter of fact, a bare closure operator is not able to tell anything about either the language of a deductive system or its type, i.~e., whether it is a Hilbert-style, equational, or Gentzen-style system. Once we look at the lattice of theories of a logic as a quantale module, rather than a mere complete lattice, we are able to capture both the deductive and syntactic parts of the deductive system at hand.

Galatos and Tsinakis' work was deepened to some extent by the present author in \cite{russothesis, russoapal}, but the development of the categorical and algebraic machinery ended up by outdoing the applications to logic. This paper's main purpose is therefore to show how to concretely apply to logical systems the order-theoretic framework initiated in \cite{galtsi}, developed mainly on its algebraic side by the present author, and occasionally mentioned in other works with reference to logic \cite{cinmor,galgil,mor,raf,fontbook}.

In particular, we are going to face several quite typical situations such as language expansion, combination of logics (about which the reader may find a different approach in \cite{sern1,sern2}), and amalgamation. In all of such cases, we shall propose different constructions making use of the theory of quantale modules as either an exclusive or an auxiliary tool, eventually comparing the results of a pure quantale-theoretic approach with the more flexible mixture of abstract logical means with order-theoretic and algebraic ones.

In the first section, we shall recall the main algebraic, categorical, and order-theoretic tools, and we will also add some new results regarding the representation of congruences of quantale modules. In Section \ref{quantlog} we shall briefly revise the main ingredients of the quantale-theoretic approach to deductive systems.

Sections from \ref{tensor} to \ref{logicamalg} contain our main results. In Section \ref{tensor} we will show how to apply the tensor product of quantale modules in order to expand a language and, above all, we will prove that the lattice of theories of the initial system remains untouched by this procedure, in the sense that such a (sup-)lattice of theories embeds in the one of the expanded system.

In Section \ref{amalgsec} we shall discuss the amalgamation property for languages and deductive systems. It is known that the amalgamation property does not hold for quantales while holds for quantale modules; nonetheless, we prove that the quantales of substitutions of propositional languages do enjoy the strong amalgamation property at least in the case of languages with a common fragment. On the deductive side, we show how to concretely carry out an amalgamated coproduct of modules of theories via the standard procedure for quantale modules. A different -- and possibly more interesting for logicians -- construction will be presented in Section \ref{logicamalg}.

Section \ref{coproddssame} contains a construction from ground up of coproducts of deductive systems. The only assumptions on the given logics is that they have to be non-trivial and of the same type, namely, both on formulas, equations, or sequents closed under the same types. Despite of not being quantale-based, the construction is made possible by a heavy use of the results recalled and those proved in Section \ref{ordtool}. In particular, we will show that the module of theories of the new system contains isomorphic copies of those of the initial ones and of their algebraic coproduct.

In Section \ref{logicamalg} we shall refine the results of Section \ref{coproddssame} to the case of amalgamation, i.~e.,  by adding another deductive system which is representable in the two given ones. Again, the logical construction works very well and the amalgamating object we found is purely logical in nature, with isomorphic copies of all the systems involved, plus their algebraic amalgamated coproduct, herein embedded. On the other hand, both the coproduct of the previous section and the amalgamated coproduct of Section \ref{logicamalg} are built on pretty large languages (the disjoint union of the initial ones), which makes the machinery flexible enough to handle outermost concrete situations with a small additional effort.

Throughout the paper, due to the quantity and quality of notations, we shall use several simplifications, such as omitting parentheses in powersets ($\wp(X)$ will be denoted by $\wp X$) and the symbol ``$\circ$'' in map compositions whenever convenient. Further abbreviations or abuse of notations will be pointed out when needed.

\section{Known and novel order-theoretic tools}
\label{ordtool}

In this section we shall briefly recall definitions and results on the ordered algebraic structures directly involved in our main results. For any further information on the topics, we refer the reader to \cite{krpa,mulvey,rosenthal} for what concerns quantales, and to \cite{paseka,mulnaw,russothesis,russojlc,russoapal,russosajl,solo} for quantale modules.

\subsubsection*{Basics}

The category $\SL$ of \emph{sup-lattices} has complete lattices as objects and maps preserving arbitrary joins -- or, which amounts to the same when the orders are complete, residuated maps -- as morphisms. The bottom element of a sup-lattice shall be denoted by $\bot$ and the top element by $\top$. We recall that any sup-lattice morphism obviously preserve the bottom element while it does not need to preserve the top. 

Quantales are often defined as sup-lattices in the category of semigroups. Since we shall only deal with unital quatales, we can say that $(Q, \bigvee, \cdot, 1)$ is a \emph{quantale} if $(Q,\bigvee)$ is a sup-lattice, $(Q, \cdot, 1)$ is a monoid, and the product is biresiduated, i.~e., for all $a,b \in Q$,
$$\exists b\under a = \max\{c \in Q \mid bc \leq a\} \text{ and } \exists a/b = \max\{c \in Q \mid cb \leq a\}.$$
The above condition is equivalent to the distributivity of $\cdot$ w.r.t. any join:
$$\forall a \in Q \ \forall B \subseteq Q \ \left(a \cdot \bigvee B = \bigvee\limits_{b \in B} (a \cdot b) \text{ and } \left(\bigvee B\right) \cdot a = \bigvee\limits_{b \in B} (b \cdot a)\right).$$
A quantale is \emph{commutative} if so is the multiplication and \emph{integral} if $1 = \top$.

The morphisms in the category $\cat Q$ of quantales are maps that are simultaneously sup-lattice and monoid homomorphisms or, that is the same, residuated monoid homomorphisms. 

The ring-like countenance of quantales obviously suggests a natural definition of module. Given a quantale $Q$, a \emph{left module over $Q$} (or, simply, {\em left $Q$-module}) is a sup-lattice $(M, \bigvee)$ acted on by $Q$ via a \emph{scalar multiplication} $\cdot: (a,u) \in Q \times M \mapsto a \cdot u \in M$ such that  
\begin{itemize}
\item $(ab) \cdot u = a \cdot (b \cdot u)$, for all $a, b \in Q$ and $u \in M$;
\item the scalar multiplication distributes over arbitrary joins in both arguments or, equivalently, is biresiduated;
\item $1 \cdot u = u$, for all $u \in M$.\footnote{Using a different symbol for this action would make the notations much heavier without helping the reading, so we rather preferred to use the same symbol of the product in the quantale, relying on the context and different sets of letters for scalars and ``vectors'' for the meaning of each of its occurrences. Whenever convenient, we shall also drop it.}
\end{itemize}
Right modules are defined analogously, mutatis mutandis. Moreover, if $R$ is another quantale, a sup-lattice $M$ is a $Q$-$R$-bimodule if it is a left $Q$-module, a right $R$-module, and in addition $(a \cdot_Q u) \cdot_R a' = a \cdot_Q (u \cdot_R a')$ for all $a \in Q$, $a' \in R$, and $u \in M$.

We also recall that, as for the quantale product, the biresiduation of $\cdot$ induces two more maps:
\begin{itemize}
\item[] $\under: (a,u) \in Q \times M \mapsto a\under u = \max\{v \in M \mid av \leq u\} \in M$, and
\item[] $/: (u,v) \in M \times M \mapsto u/v = \max\{a \in Q \mid av \leq u\} \in Q.$
\end{itemize}
We shall normally refer to ``modules'' and use the left module notation whenever a definition or a result can be stated both for left and right modules.

Given two $Q$-modules $M$ and $N$, a map $f: M \to N$ is a $Q$-module homomorphism if it is a sup-lattice homomorphism which preserves the scalar multiplication, namely, an action-preserving residuated map. For any quantale $Q$ we shall denote by $Q\Mod$ and $\rMod Q$ respectively the categories of left $Q$-modules and right $Q$-modules with the corresponding homomorphisms. Moreover, if $R$ is another quantale $Q\biMod R$ shall denote the category whose objects are $Q$-$R$-bimodules and morphisms are maps which are simultaneously left $Q$-module and right $R$-module morphisms.

\subsubsection*{Representing morphisms and congruences}

Quantale and quantale module morphisms are intimately connected with the so-called \emph{quantic nuclei} and \emph{$\Q$-module nuclei} (the latter also called \emph{structural closure operators} in \cite{galtsi} and \cite{russoapal}). A quantic nucleus over a quantale $Q$ is a closure operator $j: Q \to Q$, i.~e., a monotone, extensive, idempotent operator such that, for all $a, b \in Q$, $j(a)j(b) \leq j(ab)$. A nucleus over a $Q$-module $M$ is a closure operator $\g: M \to M$ such that, for all $a \in Q$ and $u \in M$, $a\g(u) \leq \g(au)$.

The images $Q_j$ and $M_\g$ of both quantic and $\Q$-module nuclei are closure systems of their respective domains, and therefore are closed under arbitrary meets. Moreover, $(Q_j, {}^j\bigvee, \cdot_j, j(1))$ is a quantale, with $a \cdot_j b \bydef j(ab)$ and ${}^j\bigvee A \bydef j(\bigvee A)$ ($\{a,b\} \cup A \subseteq Q_j$), and $(M_\g, {}^\g\bigvee)$ is a $Q$-module with the join defined as for $Q_j$ and the scalar multiplication $a \cdot_\g u \bydef \g(au)$ ($a \in Q$, $u \in M_\g$). By restricting the codomains of $j$ and $\g$ to their respective images, we obtain onto homomorphisms, hence $Q_j$ is homomorphic image of $Q$ and $M_\g$ is homomorphic image of $M$.

Furthermore, if $h: Q \to R$ is a quantale homomorphism with residual map $h_*: R \to Q$, then $j = h_* \circ h$ is a quantic nucleus on $Q$ and $Q_j \cong h[Q]$. Analogously, if $f: M \to N$ is a $Q$-module morphism with residuum $f_*$, then $\g = f_* \circ f$ is a $Q$-module nucleus on $M$ and  $M_\g \cong f[M]$.

It is worth recalling also that the set of quantic nuclei on a quantale $Q$ is a complete lattice whose meet is defined pointwise \cite[Proposition 3.1.3]{rosenthal}. The same holds for module nuclei, and the proof is a trivial adaptation of the one for quantic nuclei. The lattices of quantic nuclei on a quantale $Q$ shall be denoted by $\cat N(Q)$ and the one of $Q$-module nuclei on a $Q$-module $M$ by $\cat{N}_Q(M)$. As a consequence of the one-to-one correspondence between nuclei and congruences, each of them is isomorphic to the lattice of congruences of the respective structure.

Then, besides the universal algebraic corresponce between homomorphisms and congruences given by the isomorphism theorems, quantale and $\Q$-module congruences and homomorphisms can also be described by means of nuclei. But there is one more useful tool for dealing with congruences in such structures: the so-called saturated elements. 

For quantales, we hereby recall the pertinent definition and result from \cite{russosajl}.
\begin{definition}\label{satel}
Let $Q$ be a (not necessarily unital) quantale, and $\vartheta \subseteq Q^2$ be a binary relation on $Q$. An element $s$ of $Q$ is called \emph{$\vartheta$-saturated} if, for all $(a,b) \in \vartheta$ and $c,d \in Q$, the following conditions hold:
\begin{enumerate}[(i)]
\item $cad \leq s$ iff $cbd \leq s$;
\item $ac \leq s$ iff $bc \leq s$;
\item $ca \leq s$ iff $cb \leq s$;
\item $a \leq s$ iff $b \leq s$.
\end{enumerate}
We shall denote by $Q_\vartheta$ the set of $\vartheta$-saturated elements of $Q$.
\end{definition}

\begin{remark}\label{satuni}
If $Q$ is unital, conditions (ii--iv) of Definition \ref{satel} are redundant, since they are all immediate consequences of (i).
\end{remark}

\begin{theorem}\label{satquo}
Let $Q$ be a quantale, $\vartheta \subseteq Q^2$, and
$$\rho_\vartheta: a \in Q \mapsto \bigwedge\{s \in Q_\vartheta \mid a \leq s\} \in Q.$$
Then $\rho_\vartheta$ is a quantic nucleus whose image is $Q_\vartheta$. Moreover, $Q_\vartheta$, with the structure induced by $\rho_\vartheta$, is isomorphic to the quotient of $Q$ w.r.t. the congruence generated by $\vartheta$.
\end{theorem}

We shall now extend the above result to the case of quantale modules, which has not been considered in the literature so far to the best of our knowledge.
\begin{definition}\label{satelm}
Let $Q$ be a quantale, $M$ a $Q$-module, and $\vartheta$ be a binary relation on $M$. An element $s$ of $M$ is called \emph{$\vartheta$-saturated} if, for all $(v,w) \in \vartheta$ and $a \in Q$, the following condition hold:
\begin{equation}\label{sateq}
av \leq s \iff aw \leq s.
\end{equation}
We shall denote by $M_\vartheta$ the set of $\vartheta$-saturated elements of $M$.
\end{definition}

\begin{proposition}\label{satinfresm}
For any $Q$-module $M$,  and for all binary relation $\vartheta$ on it, $M_\vartheta$ is closed w.r.t. arbitrary meets. Moreover, for all $s \in M_\vartheta$ and for all $a \in Q$, $a \under s$ belong to $M_\vartheta$.
\end{proposition}
\begin{proof}
Let $X \subseteq M_\vartheta$, $(v,w) \in \vartheta$, and $a \in Q$. We have
$$av \leq \bigwedge X \iff \forall s \in X(av \leq s) \iff \forall s \in X(aw \leq s) \iff aw \leq \bigwedge X.$$

Now let $a,b \in Q$, $(v,w) \in \vartheta$, and $s \in M_\theta$. Then
$$bv \leq a \under s \iff abv \leq s \iff abw \leq s \iff bw \leq a \under s,$$
so the assertion is proved.
\end{proof}

\begin{lemma}\label{RR'm}
If $\vartheta \subseteq \eta \subseteq M^2$, then $M_\eta \subseteq M_\vartheta$.
\end{lemma}
\begin{proof}
Trivially, if $s \in M$ is $\eta$-saturated, then (\ref{sateq}) holds for all $(v,w) \in \eta$ and, therefore, for all $(v,w) \in \vartheta$. Hence $s \in M_\eta$ implies $s \in M_\vartheta$.
\end{proof}

\begin{lemma}\label{satnucm}
Let $M$ and $N$ be $Q$-modules, and $f: M \to N$ a homomorphism with residuum $f_\ast: N \to M$ and associated nucleus $\g = f_\ast \circ f$. Then $M_\g$ coincide with the set of $\ker f$-saturated elements of $M$. 
\end{lemma}
\begin{proof}
First, recall that the properties of residuated maps guarantee that, for all $u \in M$, $\g(u) = \max\{v \in M \mid f(v) \leq f(u)\}$. By definition, an element $s$ of $M$ is $\ker f$-saturated if, for all $a \in Q$ and $v, w \in M$ such that $f(v) = f(w)$, $av \leq s$ iff $aw \leq s$. Now, if $f(v) = f(w)$ and $av \leq \g(u)$ for some $u \in M$, then $f(av) \leq f \g(u) = f f_\ast f(u) = f(u)$ and therefore $f(aw) = a f(w) = a f(v) = f(av) \leq f(u)$, from which we deduce $aw \leq \g(u)$. The inverse implication is completely analogous, hence $\g(u)$ is $\ker f$-saturated, for all $u \in M$, namely, $M_\g \subseteq M_{\ker f}$.

Conversely, let $s \in M_{\ker f}$. Since $f(s) = f f_\ast f(s) = f \g(s)$, $(s,\g(s)) \in \ker f$ and therefore we have $s \leq s$ iff $\g(s) \leq s$, which implies $\g(s) \leq s$. On the other hand, $u \leq \g(u)$ for all $u \in M$, hence $s = \g(s) \in M_\g$, and the assertion is proved.
\end{proof}

\begin{theorem}\label{satquom}
Let $M$ be a $Q$-module, $\vartheta \subseteq M^2$, and
$$\rho_\vartheta: v \in M \mapsto \bigwedge\{s \in M_\vartheta \mid v \leq s\} \in M.$$
Then $\rho_\vartheta$ is a $Q$-module nucleus whose image is $M_\vartheta$. Moreover, $M_\vartheta$, with the structure induced by $\rho_\vartheta$, is isomorphic to the quotient of $M$ w.r.t. the congruence generated by $\vartheta$.
\end{theorem}
\begin{proof}
By Proposition \ref{satinfresm}, $\rho_\vartheta[M] \subseteq M_\vartheta$. On the other hand, obviously, $\rho_\vartheta(s) = s$ for all $s \in M_\vartheta$, and therefore $\rho_\vartheta[M] = M_\vartheta$. It self-evident also that $\rho_\vartheta$ is monotone, extensive, and idempotent w.r.t. composition, i.~e. it is a closure operator. So, in order to prove that $\rho_\vartheta$ is a nucleus, we only need to show that $a\rho_\vartheta(v) \leq \rho_\vartheta(av)$ for all $a \in Q$ and $v \in M$.

Let $s \in M_\vartheta$, $a \in Q$, and $v \in M$. We have $\rho_\vartheta(av) \leq s$ iff $av \leq s$ iff $v \leq a \under s$ iff $\rho_\vartheta(v) \leq a \under s$ iff $a\rho_\vartheta(v) \leq s$. Then, setting $s = \rho_\vartheta(av)$ in the above sequence of equivalences, we get $a\rho_\vartheta(v) \leq \rho_\vartheta(av)$ for all $a \in Q$ and $v \in M$.

Now, once proved that $\rho_\vartheta$ is a nucleus, we can consider $M_\vartheta$ with its $Q$-module structure induced by $\rho_\vartheta$, and we have that the mapping $v \in M \mapsto \rho_\vartheta(v) \in M_\vartheta$ is an onto homomorphism (that we will still denote by $\rho_\vartheta$). By Lemma \ref{satnucm}, we get $M/\ker\rho_\vartheta \cong M_{\rho_\vartheta} = M_\vartheta = M_{\ker\rho_\vartheta}$. Since $\vartheta \subseteq \ker\rho_\vartheta$, if $\eta$ is the congruence generated by $\vartheta$, then $\eta \subseteq \ker\rho_\vartheta$. Denote by $p_\eta$ the natural projection of $M$ over $M/\eta$ and by $\g$ the nucleus on $M$ induced by $p_\eta$. Then, by Lemma \ref{satnucm}, $M_\eta = M_\g \cong M/\eta$. Hence, by Lemma \ref{RR'm} and the first part of this proof, we obtain $M/\ker\rho_\vartheta \cong M_{\rho_\vartheta} = M_{\ker\rho_\vartheta} \subseteq M_\eta \subseteq M_\vartheta = M_{\ker\rho_\vartheta}$. The assertion follows.  
\end{proof}

\begin{remark}\label{nuclrem}
In the rest of the paper, for any given $Q$-module nucleus $\g: M \to M$, we shall denote by the same symbol also the $Q$-module homomorphism $x \in M \mapsto \g(x) \in M_\g$ from $M$ to the $\g$-closed system $M_\g$ with the structure induced by the nucleus. The codomain of the mapping or the context will always make clear how are we thinking of it in each instance.
\end{remark}

\section{Deductive systems as quantale modules}
\label{quantlog}

In this section we shall briefly recall some basic definitions about propositional deductive systems and their representation as quantale modules. We refer the reader to \cite{russothesis} for a detailed account.

Given a propositional language $\lang = (L, \nu)$, $\nu: L \to \omega$ being the arity function, and a denumerable set of variables $\Var$, the set $\fml$ of \emph{$\lang$-formulas} over $\Var$ is defined, as usual, as the term (or the absolutely free) $\lang$-algebra over $\Var$. The \emph{substitution monoid $\sfm$ over $\lang$} is the monoid of $\lang$-endomorphisms of $\fml$. We remark that, since $\fml$ is a term algebra, each substitution is completely determined by its values on the variables. Starting from formulas, it is possible to define
\begin{itemize}
\item the set $\Eq$ of $\lang$-equations as $\fml^2$, and
\item for all $T \subseteq \omega^2$, the set $\seq_T$ of \emph{sequents} closed under the \emph{types} in $T$ as $\bigcup\limits_{(m,n) \in T} \fml^m \times \fml^n$.
\end{itemize}
An \emph{inference rule} over $\lang$ is a pair $(\Phi,\psi)$ where $\Phi$ is a set of formulas, equations, or sequents of a fixed type, and $\psi$ is a single formula, equation, or sequent of the same type. Then we say that $\phi$ is \emph{directly derivable} from $\Psi$ by the rule $(\Phi,\psi)$ if there is a substitution $\sigma$ such that $\sigma \psi = \phi$ and $\sigma[\Phi] \subseteq \Psi$. An inference rule $(\Phi,\psi)$ is usually denoted by $\frac{\Phi}{\psi}$.

An \emph{axiom} in the language $\lang$ is simply a formula (or an equation, or a sequent) in $\lang$.

\begin{definition}\label{consrel}
A \emph{propositional deductive system}, or a \emph{propositional logic} for short, $S$ over a given language $\lang$, is defined by means of a (possible infinite) set of inference rules and axioms. It consists of the pair $S = (D, \vdash)$, where $\vdash$ is a subset of $\wp(D) \times D$ -- $D$ being the set of $\lang$-formulas, the one of $\lang$-equations, or a set of $\lang$-sequents closed under type -- defined by the following condition: $\Phi \vdash \psi$ iff $\psi$ is contained in the smallest set of formulas that includes $\Phi$ together with all substitution instances of the axioms of $S$, and is closed under direct derivability by the inference rules of $S$. The relation $\vdash$ is called the \emph{consequence relation} of $S$.
\end{definition}
It is well-known (see  \cite{lossuszko}) that, given a language $\lang$, any consequence relation of an $\lang$-deductive system $S = (D, \vdash)$ verifies the conditions below for all $\{\phi, \psi\}, \Phi, \Psi \in \wp D$ and, reciprocally, any subset $\vdash$ of $\wp(D) \times D$ which satisfies such conditions is the consequence relation for some deductive system over $\lang$:
\begin{itemize}
\item if $\psi \in \Phi$ then $\Phi \vdash \psi$;
\item if $\Phi \vdash \psi$ and $\Psi \vdash \phi$ for all $\phi \in \Phi$, then $\Psi \vdash \psi$;
\item if $\Phi \vdash \psi$ then $\sigma[\Phi] \vdash \sigma \psi$ for every substitution $\sigma \in \sfm$.
\end{itemize}
In addition, if the inference rules are finitary, namely, have a finite set of premises, $\vdash$ is said to be \emph{finitary} and the following holds too:
\begin{itemize}
\item if $\Phi \vdash \psi$ then $\Phi_0 \vdash \psi$ for some finite $\Phi_0 \subseteq \Phi$.
\end{itemize}
Moreover, $\vdash$ can be equivalently defined as a binary relation on $\wp D$ satisfying the following conditions for all $\Phi, \Psi, \Xi \in \wp D$:
\begin{itemize}
\item if $\Psi \subseteq \Phi$, then $\Phi \vdash \Psi$;
\item if $\Phi \vdash \Psi$ and $\Psi \vdash \Xi$, then $\Phi \vdash \Xi$;
\item $\Phi \vdash \bigcup_{\Phi \vdash \Psi} \Psi$;
\item $\Phi \vdash \Psi$ implies $\sigma[\Phi] \vdash \sigma[\Psi]$ for each substitution $\sigma \in \sfm$.
\end{itemize}
With such a definition, $\vdash$ is finitary if and only if, for all subsets $\Phi, \Psi$ of $D$, with $\Psi$ finite, $\Phi \vdash \Psi$ implies that $\Phi_0 \vdash \Psi$ for some finite $\Phi_0 \subseteq \Phi$.

The set $\{\Phi \in \wp D \mid \forall\Psi(\Phi \vdash \Psi \To \Psi \subseteq \Phi)\}$ is a closure system of the lattice $\wp D$, hence a complete lattice, usually called the \emph{lattice of theories} of the system, and denoted by $\th_\vdash$ or, simply, $\th$ when there is no danger of confusion. It is worthwhile remarking that, in the case of an equational system with corresponding algebraic variety $\cat V$, such a lattice is isomorphic to the lattice of fully invariant congruences on the free algebra over $\omega$ generators in $\cat V$.

Starting from the approach to deductive systems by means of consequence operators, which goes back at least to Taski's work, Galatos and Tsinakis \cite{galtsi} presented a representation of them which uses complete posets acted on by complete residuated partially ordered monoids. Such a representation was eventually reformulated in terms of modules over quantales and further investigated in \cite{russothesis} and \cite{russoapal}.

The simple yet successful idea of Galatos and Tsinakis started from the observation that order theory alone is not sufficient to fully describe the complexity of a deductive system. Indeed, although certain closure operators are able to describe a consequence relation, they cannot take into account the language, while it is possible to address this issue by adding an ``algebraic side''.

More precisely, if $D$ is $\fml$, $\Eq$, or any $\seq_T$, we have a left monoid action from $\sfm$ to $D$. By applying the left adjoints to the forgetful functors from quantales to monoids and from sup-lattices to sets respectively, we obtain a quantale $\Sfm$ and a left $\Sfm$-module $\wp D$. Thanks to this change of perspective, one can see a consequence relation $\vdash$ as a $\Sfm$-module nucleus on $\wp D$, as the following result shows
\begin{proposition}[Lemma 3.5 -- \cite{galtsi}]
For any consequence relation $\vdash$ on $\wp D$, the mapping
$$\g_\vdash: \Phi \in \wp D \mapsto \bigcup_{\Phi \vdash \Psi} \Psi \in \wp D$$
is a $\Sfm$-module nucleus. Reciprocally, for any $\Sfm$-module nucleus $\g$ on $\wp D$, the relation
$$\Phi \vdash_\g \Psi \iff \Psi \subseteq \g(\Phi)$$
is a consequence relation on $\wp D$.  

Moreover, $\vdash_{\g_\vdash} = \ \vdash$ and $\g_{\vdash_\g} = \g$, for any consequence relation $\vdash$ and for any nucleus $\g$.
\end{proposition}
Given a consequence relation $\vdash$ with associated nucleus $\g$, the $\g$-closed system $\wp D_\g$ coincide with the lattice of theories $\th$ of $\vdash$. Thanks to the previous result, we can think of a consequence relation either as a binary relation or as a $\Sfm$-module nucleus, and we shall use either one of the notations depending on convenience. Similarly, we shall indifferently denote the lattice of theories by $\wp D_\g$, $\th_\vdash$ or $\th_\g$. 

\begin{remark}\label{nontrivial}
In the rest of the paper, we shall always assume the consequence relations to be non-trivial, i.~e., such that $\g(\varnothing) \neq D$ and $\g(\{x\}) = \{x\} \cup \g(\varnothing)$ for all $x \in \Var$.
\end{remark}

We conclude this section by recalling the following relevant notations and results from \cite{galtsi}.

Given a propositional language $\lang$ and $x, y, x_1, \ldots, x_{m+n} \in \Var$, let $\{V_x, V_y\}$ and $\{V_1,\ldots, V_{m+n}\}$ be partitions of $\Var$. Further, let us denote, respectively, by $\kappa_x$ the unique substitution which sends every variable to $x$, by $\kappa_{x \approx y}$ the one which sends every element of $V_x$ to $x$ and every variable in $V_y$ to $y$, and by $\kappa_{(x_1, \ldots, x_{m+n})}$ the substitution sending each variable in $V_i$ to $x_i$, for $i = 1, \ldots, m + n$. Then Corollary 5.9 and Theorem 5.13 from \cite{galtsi} can be stated as follows.
\begin{theorem}\label{kappax}
The $\Sfm$-module $\Fml$ is generated by $\{x\}$ and is isomorphic to the $\Sfm$-submodule of $\Sfm$ generated by $\{\kappa_x\}$.

Analogously, $\wp\Eq$ is generated by $x \approx y$ and is isomorphic to the $\Sfm$-submodule of $\Sfm$ generated by $\{\kappa_{x \approx y}\}$, and $\wp \seq_T$ is generated by the set $\{x_1, \ldots, x_m \To x_{m+1}, \ldots, x_{m+n} \mid (m,n) \in T\}$ and is isomorphic to the coproduct\footnote{Products and coproducts in $Q\Mod$ have the same object, namely, the Cartesian product with componentwise operations \cite[Proposition 4.2.3]{russothesis}.}  of the $\Sfm$-submodules of $\Sfm$ generated by $\{\kappa_{(x_1, \ldots, x_{m+n})}\}$, for $(m,n) \in T$.

Moreover, all of such modules are projective.
\end{theorem}

\section{Recovering a module of theories as a fragment after a language expansion}
\label{tensor}

It was proved in \cite{russoapal} that each quantale morphism $h: Q \to R$ induces an adjoint and co-adjoint functor $( \ )_h: R\Mod \to Q\Mod$, whose left adjoint is $R \tensor_Q \underline{\ \ }$. For details on the construction of the tensor product of quantale modules the reader may refer to \cite[Theorem 6.3]{russocorr}; we report here its assertion for the reader's convenience.
\begin{theorem}\label{tensormqexists}
Let $M_1$ be a right $Q$-module and $M_2$ a left $Q$-module. Then the tensor product $M_1 \tensor_Q M_2$ of the $Q$-modules $M_1$ and $M_2$ exists. It is, up to isomorphisms, the quotient $\wp(M_1 \times M_2)/\theta_R$ of the free sup-lattice generated by $M_1 \times M_2$ with respect to the (sup-lattice) congruence relation generated by the set
\begin{equation}\label{R}
\rho = \left\{
	\begin{array}{l}
		\left(\left\{\left(\bigvee X, y\right)\right\}, \bigcup_{x \in X}\{(x,y)\}\right) \\
		\left(\left\{\left(x, \bigvee Y\right)\right\}, \bigcup_{y \in Y}\{(x,y)\}\right) \\
		\left(\{(x \cdot_1 a, y)\}, \{(x,a \cdot_2 y)\}\right) \\
	\end{array} \right\vert
	\left.
	\begin{array}{l}
	X \subseteq M_1, y \in M_2 \\
	Y \subseteq M_2, x \in M_1 \\
	a \in Q \\
	\end{array}
	\right\}.
\end{equation}
\end{theorem}

In the same work, the author suggested that the tensor product could be used in order to expand the language of a deductive system, but no details were given. It must be mentioned that, given quantales $Q \leq R$ and a $Q$-module $M$, the mapping $x \in M \mapsto 1 \tensor x \in R \tensor_Q M$ is a $Q$-module morphism which, however, need not be an embedding in the general case.

We shall now prove that, in the particular case of the module of theories $\th$ of a propositional logic whose language is being expanded via a tensor product, such a morphism is indeed an embedding, and therefore $\th$ turns out to be a $\Sfm$-submodule of such a tensor product. 

Let $\lang_1$ be an expansion of a propositional language $\lang$, and $i: \Sfm \to \Sfmm$ be the associated quantale embedding. Further, let us consider an $\lang$-deductive system $(D, \vdash)$ with associated nucleus $\g$ and $\Sfm$-module of theories $\th$.

By Theorem \ref{kappax}, $\Fmll$ is isomorphic to the $\Sfmm$-module $\Sfmm\cdot\{\kappa_x\}$, whence $\Fml$ can be identified with the sup-sublattice $i[\Sfm\cdot \{\kappa_x\}]$ of $\Sfmm\cdot\{\kappa_x\}$, and therefore with a sup-sublattice of $\Fmll$. Similar considerations can be done for sets of equations and sequents. So, if $D$ is the domain of our deductive system on $\lang$ and $D_1$ the domain of the same type in $\lang_1$, $\wp D$ can be identified with a sup-sublattice of $\wp D_1$ which shall be denoted by $i[\wp D]$. With an abuse of notation, we will also denote by $i(\Psi)$ the element of $i[\wp D]$ corresponding to each element $\Psi$ of $\wp D$.

\begin{proposition}\label{generator}
We the above notations, there exists a consequence relation $\vdash_1$ on $\wp D_1$ whose $\Sfmm$-module of theories $\th_1$ is isomorphic to $\Sfmm \tensor_{\Sfm} \th$.
\end{proposition}
\begin{proof}
It suffices to observe that, according to the remarks preceding Theorem 6.7 of \cite{russoapal}, $\Sfmm \tensor_{\Sfm} \th$ is generated by $1 \tensor \g(x)$ and, therefore, is homomorphic image of $\wp D_1$. 
\end{proof}

Now, for all $\Psi \in \wp D$, let us denote by $S_\Psi$ the set of all substitutions $\sigma \in \sfmm$ such that $\sigma i(\Psi) \in i[\wp D]$:
$$S_\Psi = i(D)/i(\Psi) = \{\sigma \in \sfmm \mid \sigma i(\Psi) \in i[\Fml]\} \in \Sfmm.$$

With the next theorem, we prove that formally expanding the language of a deductive system by means of the tensor product yields a new system whose module of theories contains an isomorphic copy of $\th$.
\begin{theorem}\label{tensembed}
$\th$ is isomorphic to a $\Sfm$-submodule of $(\Sfmm \tensor_{\Sfm} \th)_i$.
\end{theorem}
\begin{proof}
Let, for all $\Phi \in \th$, $\ov\Phi$ be the following element of $\wp(\Sfmm \times \th)$:
\begin{equation}\label{ovphi}
\ov\Phi = \{(\Omega, \Psi) \in \Sfmm \times \wp\th \mid \Omega \subseteq S_\Psi \ \& \ i^{-1}[\Omega \cdot i(\Psi)] \subseteq \Phi)\}.
\end{equation}

Recalling that sup-lattices can be seen also as modules over the two-element quantale $\{0,1\}$, we shall prove that, for each $\Phi \in \th$, $\ov\Phi$ is a saturated element of the relation defined as in (\ref{R}) by proving that condition (\ref{sateq}) is verified for all the pair types in (\ref{R}). The scalar $a$ in (\ref{sateq}) shall be dropped because it is actually $\{\id\}$, while the case of $a = \varnothing$ is trivial. 

Let us check (\ref{sateq}) for the first type of pairs in (\ref{R}). Let $X \subseteq \Sfmm$ and $\Psi \in \th$; then
$$\begin{array}{l}
\{(\bigvee X, \Psi)\} \subseteq \ov\Phi \\
\iff \bigvee X = \bigcup\limits_{\Omega \in X} \Omega \subseteq S_\Psi \text{ and }  i^{-1}[\bigvee X \cdot i(\Psi)] \subseteq \Phi \\
\iff \forall \Omega \in X, \ \Omega \subseteq S_\Psi \text{ and } i^{-1}[\Omega \cdot i(\Psi)] \subseteq \Phi \\
\iff \bigcup_{\Omega \in X}\{(\Omega, \Psi)\} \subseteq \ov\Phi.
\end{array}.$$

For what concerns the second type, let $\Omega \in \Sfmm$ and $Y \subseteq \th$. Then we have:
$$\begin{array}{l}
\{(\Omega,\bigvee Y)\} \subseteq \ov\Phi \\
\iff \left[\begin{array}{l}
\Omega \subseteq S_{\bigvee Y} \\
\text{ and } \\
i^{-1}[\Omega \cdot i(\bigvee Y)] = i^{-1}[\Omega \cdot i(\bigcup\limits_{\Psi \in Y}\Psi)] = \bigcup\limits_{\Psi \in Y} i^{-1}[\Omega \cdot i(\Psi)]\subseteq \Phi 
\end{array}\right.\\
\iff \Omega \subseteq S_{\bigcup_{\Psi \in Y} \Psi} \text{ and } \forall \Psi \in Y \ (i^{-1}[\Omega \cdot i(\Psi)] \subseteq \Phi)\\
\iff \forall \Psi \in Y \ (\Omega \subseteq S_\Psi \text{ and } i^{-1}[\Omega \cdot i(\Psi)] \subseteq \Phi)\\
\iff \forall \Psi \in Y \ (\{(\Omega, \Psi)\}  \subseteq \ov\Phi)\\
\iff \bigcup_{\Psi \in Y}\{(\Omega, \Psi)\}  \subseteq \ov\Phi.
\end{array}$$

Last, let us consider $\Omega \in \Sfmm$, $\Sigma \in \Sfm$, and $\Psi \in \wp D$. $\Omega \subseteq S_{\Sigma \cdot \Psi}$ if and only if $\Omega \cdot i(\Sigma \cdot \Psi) \in i[\wp D]$; on the other hand, $\Omega \cdot i(\Sigma \cdot \Psi) = (\Omega \cdot i(\Sigma))\cdot i(\Psi)$, whence $\Omega \in S_{\Sigma \cdot \Psi}$ if and only if $\Omega \cdot i(\Sigma) \in S_\Psi$. Moreover, the same equality $\Omega \cdot i(\Sigma \cdot \Psi) = (\Omega \cdot i(\Sigma))\cdot i(\Psi)$ implies that $i^{-1}[\Omega \cdot i(\Sigma \cdot \Psi)] \subseteq \Phi$ if and only if $i^{-1}[(\Omega \cdot i(\Sigma))\cdot i(\Psi)] \subseteq \Phi$, and therefore $\{\Omega, \Sigma \cdot \Psi)\} \subseteq \ov\Phi$ if and only if $\{(\Omega \cdot i(\Sigma), \Psi)\} \subseteq \ov\Phi$. This finally proves that $\ov\Phi$ is saturated in $\Sfmm \times \th$ w.r.t. the relation which determines the tensor product.

As a final step, we remark that the mapping $\Phi \mapsto \ov\Phi$ is obviously injective, so the assertion follows.
\end{proof}

\section{Amalgamating languages and modules of theories}
\label{amalgsec}

Let $\cat K$ be a class of algebras. An \emph{amalgam} (or a \emph{V-formation}) in $\cat K$ is a 5-tuple $(A, f, B, g, C)$, where $A, B, C \in \cat K$ and $f: A \lto B$, $g: A \lto C$ are injective homomorphisms. An amalgam $(A, f, B, g, C)$ is said to be \emph{embeddable} if there exist an object $D$ and two injective homomorphisms $f': B \lto D$ and $g': C \lto D$ such that $f' \circ f = g' \circ g$. It is \emph{strongly embeddable} if, in addition, $f'[B] \cap g'[C] = f'f[A] = g' g[A]$.

\begin{definition}\label{ama}
We say that a class of algebras $\cat K$ has the \emph{amalgamation property} (resp.: \emph{strong amalgamation property}) if all amalgams in $\cat K$ are embeddable (resp.: \emph{strongly embeddable}).
\end{definition}

The amalgamation property can be defined in the more general setting of category theory (see \cite{hofmis,tho}) in such a way that it basically reduces to Definition \ref{ama} in the case of algebraic categories. 

It was shown in \cite{krpa} that quantales do not enjoy the amalgamation property, but this is actually an obvious consequence of the failure of the amalgamation property for semigroups and monoids \cite{how,kim,ren}. Indeed, given a non-embeddable monoid amalgam $\mathfrak A = (A, f, B, g, C)$, applying the powerset functor we get a quantale amalgam $\wp\mathfrak A = (\wp A, \wp f, \wp B, \wp g, \wp C)$ whose embeddability would imply the embeddability of $\mathfrak A$ by applying the forgetful functor and restricting all the morphisms of $\wp\mathfrak A$ to the submonoids of singletons.

However, we shall see in the present section that monoids and quantales of substitutions of propositional languages do actually enjoy amalgamation.

Let $\lang_1 = (L_1, \nu_1)$ and $\lang_2 = (L_2, \nu_2)$ be propositional languages with a common fragment $\lang = (L_1 \cap L_2, \nu)$, where $\nu = {\nu_1}_{\restr L_1\cap L_2} = {\nu_2}_{\restr L_1 \cap L_2}$. Hence we have the amalgam of monoids
\begin{equation}\label{vmondiag}
\begin{tikzcd}
\sfmm & & \sfmmm \\
& \sfm \arrow[ul, ,"i_1"] \arrow[ur, "i_2"'] &\\
\end{tikzcd}
,\end{equation}
where, for any substitution $\sigma$ in $\lang$, $i_1(\sigma)$ and $i_2(\sigma)$ are the substitutions, in $\lang_1$ and $\lang_2$ respectively, such that $\sigma(x) = i_1(\sigma)(x) = i_2(\sigma)(x)$ for all $x \in \Var$.

\begin{proposition}\label{amalgmon}
The amalgam (\ref{vmondiag}) is strongly embeddable in $\sfmu$.
\end{proposition}
\begin{proof}
First of all, let us recall that $\fml = \fmll \cap \fmlll$, $\fmll \subseteq \fmlu$, and $\fmlll \subseteq \fmlu$.

Let  $j_1: \sfmm \to \sfmu$ and $j_2: \sfmmm \to \sfmu$ be the maps defined as follows: for all $\sigma \in \sfmm$ and $\sigma' \in \sfmmm$, $j_1(\sigma)$ and $j_2(\sigma')$ are the unique substitutions in $\lang_1 \cup \lang_2$ such that $j_1(\sigma)(x) = \sigma(x)$ and $j_2(\sigma')(x) = \sigma'(x)$ for every $x \in \Var$. Both $j_1$ and $j_2$ are obviously one-to-one monoid homomorphisms. Moreover, for all $\sigma \in \sfmm$ and $\sigma' \in \sfmmm$, $j_1(\sigma) = j_2(\sigma')$ iff $\sigma_{\restr\Var} = \sigma_{\restr\Var}'$ iff $\sigma = \sigma' \in \sfm$, whence $j_1 \circ i_1 = j_2 \circ i_2$ and $j_1[\sfmm] \cap j_2[\sfmmm] = j_1 i_1[\sfm] = j_2 i_2[\sfm]$. 
\end{proof}

\begin{corollary}\label{amalgquant}
The quantale $\Sfmu$ is the strong amalgamated coproduct of $\Sfmm$ and $\Sfmmm$ w.r.t. $\Sfm$.
\end{corollary}

The strong amalgamation property was proved to hold for quantale modules by Nkuimi-Jugnia \cite{nku}; eventually, two easier proofs were presented independently \cite{russothesis,solo}. We refer to any of the cited works for a detailed proof of the property, but we want to recall here how the amalgam embeddings occur (in the case of $Q\Mod$, as usual).

Let $M$, $N$ and $P$ be $Q$-modules, and $f: P \to M$ and $g: P \to N$ be two injective homomorphisms. Then the amalgamating object is the quotient of the coproduct $M \times N$ w.r.t. the $Q$-module congruence $\vartheta$ generated by the set $\{((f(w),\bot_N),(\bot_M,g(w))) \mid w \in P\}$. The associated embeddings of $M$ and $N$ are defined as follows:
$$f': v \in M \lmapsto (v,\bot)/\vartheta \in A \quad \textrm{ and } \quad g': w \in N \lmapsto (\bot, w)/\vartheta \in A.$$
As observed in \cite{russothesis}, $\vartheta$ can be described as follows:
\begin{equation}\label{simama}
(u,v) \vartheta (u',v') \quad \iff \quad \exists w, w' \in P: \ \left\{\begin{array}{l} f(w) = u \\ f(w') = u' \\ g(w) = v' \\ g(w') = v\end{array}\right..
\end{equation}
We also have sufficient information for characterizing the $\vartheta$-saturated elements of $M \times N$.

\begin{proposition}\label{satamalg}
With the notations introduced above, for all $(u,v) \in M \times N$, $(u,v)$ is $\vartheta$-saturated if and only if $[\bot_M,u] \cap f[P] = [\bot_N,v] \cap g[P].$
\end{proposition}
\begin{proof}
In the present case, condition (\ref{sateq}) reads as follows: for all $a \in Q$ and for all $w \in P$, $a(f(w),\bot_N) \leq (u,v) \iff a(\bot_M, g(w)) \leq (u,v)$. Such an equivalence holds if and only if $\forall a \in Q \ \forall w \in P \ (f(aw) \leq u \iff g(aw) \leq v)$, i.~e., if and only if
$$\{w \in P \mid f(w) \leq u\} = \{w \in P \mid g(w) \leq v\},$$
which is exactly what we wanted to prove.
\end{proof}

Since sup-lattices can be identified with modules over the quantale structure associated to the Boolean algebra $\{0,1\}$ (with the obvious scalar multiplication) along with their morphisms, from the amalgamation property of quantale modules readily follows that sup-lattices enjoy the strong amalgamation property too.

Next, we will show how the results of the previous sections and from \cite{russoapal,russosajl} concretely apply to the problem of combining two different deductive systems in a single one. Standard methods for merging different deductive systems may be useful in various applications in automated reasoning, such as automated theorem provers or decision-making processes. In this section we shall simply discuss the standard construction of an amalgamating module in the special case of a V-formation of deductive systems. In the next sections we will propose alternative yet more ``concrete'' (as far as Abstract Logic can be concrete) constructions and we shall compare them with the standard one.

Let us consider three deductive systems $(D, \vdash)$, $(D_1, \vdash_1)$, and $(D_2, \vdash_2)$ over the same language $\lang$ -- whose corresponding $\Sfm$-modules of theories shall be denoted, respectively, by $\th$, $\th_1$, and $\th_2$ -- such that there exist $\Sfm$-module embeddings $r_1: \th \to \th_1$ and $r_2: \th \to \th_2$. Let also $\g$, $\g_1$, and $\g_2$ be the nuclei associated, respectively, to the three systems. We recall that, according to Remark \ref{nuclrem}, we shall use the same names for the morphisms obtained from such maps by restricting the codomains to the respective images. 

In such a situation there exist a $\Sfm$-module $M$ and embeddings $n_1: \th_1 \to M$ and $n_2: \th_2 \to M$ such that $n_1 \circ r_1 = n_2 \circ r_2$ and $n_1[\th_1] \cap n_2[\th_2] = n_1 r_1(\th) = n_2 r_2(\th)$. Moreover, the projectivity of the modules $\wp D, \wp D_1, \wp D_2$, and $\wp D_1 \amalg \wp D_2$ (see Theorem \ref{kappax}) guarantees the existence of morphisms $f_1$ and $f_2$ which make the following diagram commute ($g_1$ and $g_2$ are the canonical embeddings). 
\begin{equation}\label{superdiag}
\begin{tikzcd}
 & \wp D_1 \amalg \wp D_2  \arrow[dd, "g_2", near end, hookleftarrow] \arrow[rr, two heads, "\delta"', swap] &  & M & \\
\wp D_1 \arrow[rr, two heads, "\g_1", crossing over,  near end] \arrow[ur, hook, "g_1"] &  & \th_1 \arrow[ur, hook, "n_1", swap] &  & \\ 
& \wp D_2  \arrow[rr, two heads, "\g_2", swap, near start] &&\th_2 \arrow[uu, hook, "n_2", swap] &\\
  \wp D \arrow[rr, two heads, "\g", swap] \arrow[uu, "f_1"] \arrow[ur, "f_2"] &  &  \th \arrow[ur, hook, "r_1", swap] \arrow[uu, crossing over, near start, hook, "r_2", swap] &  \\
\end{tikzcd}
\end{equation}

The commutativity of the diagram above is just a direct application of the categorical and algebraic machinery of \cite{galtsi,russoapal}; according to the construction of the amalgamating module, $M$ is the lattice of theories of a deductive system over a domain of two-sorted syntactic constructs.

More precisely, let us think of $D, D_1$ and $D_2$ as sets of sequents closed under the types $T, T_1, T_2 \subseteq \omega^2$ respectively.\footnote{We recall that $\fml$ and $\Eq_\lang$ can be thought of as the sets of, respectively, $\{(0,1)\}$-sequents and $\{(1,1)\}$-sequents.} Then $M$ will be the module of theories of a system defined on $\wp D_1 \times \wp D_2$, i.~e., on pairs made of a set of $T_1$-sequents and one of $T_2$-sequents. According to (\ref{simama}), two of such pairs -- say $(\Phi_1, \Phi_2)$ and $(\Psi_1,\Psi_2)$ generate the same theory in $M$ if and only if there exist $\Xi$ and $\Xi'$ in $\wp D$ such that $r_1 \g(\Xi) = \g_1(\Phi_1)$, $r_2 \g(\Xi) = \g_1(\Psi_2)$, $r_1 \g(\Xi') = \g_1(\Psi_1)$, and $r_2 \g(\Xi') = \g_1(\Phi_2)$.

A system like that may not look very ``concrete'', in the sense that consequence relations on pairs of sets of sequents are not common at all. However, both the positive and negative results of the theory must be interpreted, in our opinion, as a road map toward the concrete solutions to several problems pointing out the right directions, possible obstacles, and blind alleys. For example, the amalgamation property for quantale modules, and the way it works, guarantees that it is possible to merge deductive systems and that the result is something which is not really far from a logical system in the classical meaning. On the other hand, it is a general construction encompassing any module amalgam on any quantale, and therefore it makes sense to seek for alternative yet equally good constructions based on the specific case of deductive systems. 

In the next sections we shall prove that more handy constructions are indeed possible, and quantale modules continue playing an extremely relevant role in proving the good properties of such alternative amalgamations.

\section{Logical coproducts of deductive systems}
\label{coproddssame}

In the present section we shall describe how to build a deductive system which includes two given ones of the same type without assuming the possibility of a common subsystem. Here we will define the new system in a pretty classical way, namely, by means of axioms and inference rules, and we will use quantales and modules only as tools for proving that the result of the construction has the most desirable properties.

More in details, we shall first consider the disjoint union of the languages of the two initial systems and the domain of the same syntactic constructs of such systems in the new language, then we will extend the two consequence relations to such a domain, and finally we shall define a new logic by means of all of the axioms and rules of the two given ones. Once the new deductive system will be defined, we will prove that the two initial consequence relations are fully represented inside the new one. Moreover, we will also show that each consequence relation is representable in its corresponding extension and that each of the domains of the initial systems, once embedded in the new domain, remains totally untouched by the extension of the consequence relation of the other one. Last, we will show that each extension of the two initial systems is isomorphic to the one obtained by extending the scalars using the tensor product, as in Section \ref{tensor}.

Let $(D_1, \vdash_1)$ and $(D_2, \vdash_2)$ be two deductive systems of the same type over the languages $\lang_1$ and $\lang_2$ respectively, with associated nuclei $\g_1$ and $\g_2$ and modules of theories $\th_1 = (\wp D_1)_{\g_1}$ and $\th_2 = (\wp D_2)_{\g_2}$ respectively, and let $\lang$ be the disjoint union of $\lang_1$ and $\lang_2$. Then $\Sfm$ is the coproduct of the quantales $\Sfmm$ and $\Sfmmm$ and, consequently, every $\Sfm$-module is also a $\Sfmm$- and a $\Sfmmm$-module. Let us consider the $\Sfm$-module $\wp E$ of the same syntactic constructs of the two given systems; clearly $\wp D_1$ and $\wp D_2$ are contained in $\wp E$ -- let us denote by $d_1$ and $d_2$ the respective inclusion maps -- and we can define the ${\Sfm}_i$-module nuclei
$$\g_i': \Phi \in \wp E \mapsto \g_i(\Phi \cap D_i) \cup (\Phi \setminus D_i) \in \wp E, \quad i = 1,2.$$
Now consider the following $\Sfm$-nuclei on $\wp E$, for $i = 1,2$:
\begin{enumerate}[(i)]
\item $\delta_i$ is the nucleus associated to the consequence relation $\vdash_{\delta_i}$ on $E$ defined by means of the axioms and rules of $\vdash_i$;
\item $\delta = \delta_1 \vee \delta_2 = \bigwedge\{\rho \in \cat{N}_{\Sfm}(E) \mid \delta_1 \leq \rho \text{ and } \delta_2 \leq \rho\}$.
\end{enumerate}
Before continuing, let us remark some facts about such nuclei and their corresponding consequence relations. First of all, it is worth noticing that the infimum of nuclei is nothing else than the pointwise intersection in this case.

For each $i$, $\delta_i = \bigwedge \{\rho \in \cat{N}_{\Sfm}(E) \mid \ \forall \Phi \in \wp E (\rho(\Phi) \supseteq \g_i'(\Phi))\}$, namely, it is the smallest nucleus on $\wp E$ for which $\g_i'(\Phi) \subseteq \delta_i (\Phi)$ for all $\Phi \in \wp D_i$. 

Last, the consequence relation corresponding to $\delta$ can be easily described as the one defined by the union of the axioms and rules of $\g_1$ and $\g_2$.

From now on, in this section, the sets of axioms of $\g_1$ and $\g_2$ will be denoted by $\ax_1$ and $\ax_2$ respectively. Moreover, for $i=1,2$, $\th_{\delta_i}$ shall denote the $\Sfm$-module of theories of $\delta_i$, and $\th$ the one of $\delta$.

The following lemma, which plays a fundamental role in most of the results of both this and next sections, basically asserts that the consequence relation $\vdash_{\d_i}$ acts trivially on $\wp D_k$, for $i \neq k$. 

\begin{lemma}\label{deltapsi}
Let $i \neq k \in \{1,2\}$. Then, for all $\Phi \cup \{\psi\} \in \wp D_k$, $\psi \in \delta_i(\Phi)$ if and only if $\psi \in \Phi$.
\end{lemma}
\begin{proof}
The right-to-left implication is trivial. Recalling that $\psi \in \delta_i(\Phi)$ if and only if $\Phi \vdash_{\delta_i} \psi$, by Definition \ref{consrel} we have that $\psi \in \delta_i(\Phi)$ if and only if $\psi$ satisfies any of the following conditions:
\begin{itemize}
\item $\psi$ is a theorem of $\vdash_{\delta_i}$,
\item $\psi$ is derivable from $\Phi$ and the axioms of $\vdash_{\delta_i}$ via inference rules of the system, or
\item $\psi \in \Phi$.
\end{itemize}
It is easy to see that an element of $D_k$ can be obtained by applying a substitution to some element of $E$ only if the latter is itself in $D_k$. For this reason, no element of $D_k$ can be a theorem of $\vdash_{\delta_i}$.

On the other hand, an element $\eta$ of $E$ can be inferred from $\Phi$ via a rule $\frac{\Xi}{\chi}$ if and only if there exists a substitution $\sigma$ such that $\sigma[\Xi] \subseteq \Phi \cup \ax_i$ (with $\sigma[\Xi] \cap \Phi \neq \varnothing$, in order to exclude the previous case) and $\sigma(\chi) = \eta$. On its turn, this implies that the nonempty subset $\sigma^{-1}[\sigma[\Xi] \setminus \ax_i]$ of $\Xi$ is contained in $D_k$ and, therefore, in $D_i \cap D_k = \Var$, because only elements of $D_i$ appear in the rules of $\vdash_{\delta_i}$. Then, by the hypothesis of non-trivialness of the systems $(D_1, \g_1)$ and $(D_2, \g_2)$ (see Remark \ref{nontrivial}), $\chi$ must necessarily be one of the variables belonging to $\sigma^{-1}[\sigma[\Xi] \setminus \ax_i]$, whence $\eta \in \Phi$. It follows that nothing but the theorems of $\vdash_{\delta_i}$ and the elements of $\Phi$ themselves are derivable from $\Phi$ and $\ax_i$.

Therefore $\psi \in \delta_i(\Phi)$ implies $\psi \in \Phi$, quod erat demonstrandum.
\end{proof}

Now, recalling that the $\Sfm$-module of theories $\th = \wp E_\delta$ becomes a $\Sfmm$-module and a $\Sfmmm$-module under the functors induced by the canonical embeddings of, respectively, $\Sfmm$ and $\Sfmmm$ into $\Sfm$, we get the following

\begin{proposition}\label{gammaiembedlem}
There exist ${\Sfm}_i$-module embeddings of $\th_i$ and ${\Sfm}_k$-module embeddings of $\wp D_k$ into $\th_{\delta_i}$, $i, k \in \{1,2\}$ with $i \neq k$.
\end{proposition}
\begin{proof}
Let us define, for distinct $i$ and $k$ in $\{1,2\}$, the following maps:
$$f_i: \Phi \in \th_i \mapsto \delta_i(\Phi) \in \th_{\delta_i} \quad \text{and} \quad g_k: \Phi \in \wp D_k \mapsto \delta_i(\Phi) \in \th_{\delta_i}.$$
For all $\Phi \in \th_i$, $\Phi = \g_i'(\Phi) = \delta_i(\Phi) \cap D_i$. If $\Psi \in \th_i \setminus \{\Phi\}$, without losing generality, we can assume that there exists $\phi \in \Phi \setminus \Psi$. Then
$$\phi \in \Phi \setminus \Psi = (\delta_i(\Phi) \cap D_i) \setminus (\delta_i(\Psi) \cap D_i) \subseteq \delta_i(\Phi) \setminus \delta_i(\Psi),$$
whence $\delta_i(\Phi) \neq \delta_i(\Psi)$. Then the maps $f_i$ are injective. The fact that they preserve the action from ${\Sfm}_i$ and arbitrary joins is trivial.

The situation for the maps $g_k$ is analogous, the injectivity being a direct consequence of Lemma \ref{deltapsi}.
\end{proof}

We are now ready to prove that each of the the modules of theories $\th_i$ is embeddable as a ${\Sfm}_i$-module in $\th$. By \cite[Theorem 7.1]{russoapal}, this implies that each of the systems $(D_i, \vdash_i)$ is representable in $(E, \vdash_\d)$.

\begin{theorem}\label{gammaiembedth}
There exist ${\Sfm}_i$-module embeddings of $\th_i$, $i = 1, 2$, into $\th$.
\end{theorem}
\begin{proof}
It is immediate to verify that $\delta$ is the nucleus associated to the congruence generated by the relation
$$R_\delta = \{(\Phi, \delta_i(\Phi)) \mid \Phi \in \wp D_i, i = 1, 2\}.$$
We shall prove that $\delta_i(\Phi)$ is $R_\delta$-saturated for all $\Phi \in \wp D_i$ and for $i=1,2$.

Let $\Sigma \in \Sfm$, $\Psi \in \wp D_k$ and $\Phi \in \wp D_i$, for $i = 1, 2$. For $k = i$, it is obvious that $\Sigma \cdot \Psi \subseteq \delta_i(\Phi)$ if and only if $\Sigma \cdot \delta_i(\Psi) \subseteq \delta_i(\Phi)$. If $k \neq i$, then $\Psi \subseteq \delta_k(\Psi)$ guarantees that $\Sigma \cdot \delta_k(\Psi) \subseteq \delta_i(\Phi)$ implies $\Sigma \cdot \Psi \subseteq \delta_i(\Phi)$. For the converse implication, let us observe that, by Proposition \ref{gammaiembedlem}, $\Sigma \cdot \Psi \subseteq \delta_i(\Phi)$ implies that $\Sigma \cdot \Psi \in D_i$, whence $\Psi \subseteq \Var$. Then $\delta_k(\Psi) = \Psi$ and therefore $\Sigma \cdot \Psi \subseteq \delta_i(\Phi)$ implies $\Sigma \cdot \delta_k(\Psi) \subseteq \delta_i(\Phi)$.

Now that we know that the sets of the form $\delta_i(\Phi)$ are $R_\delta$-saturated and, therefore, $\delta$-closed, we get immediately two injective maps
$$e_i: \g_i(\Phi) \in \th_i \mapsto \delta_i(\Phi) \in \th,$$
which are ${\Sfm}_i$-module embeddings by Proposition \ref{gammaiembedlem}. The assertion is proved.
\end{proof}

Next, we show that the language expansion by means of the tensor product, as shown in Section \ref{tensor}, does not yield an unserviceable abstract object; the resulting module is indeed the module of theories of the most natural possible extension of the initial consequence relation to a domain over a richer language. Indeed, we show that $\th_{\d_1}$ and $\th_{\d_2}$ are nothing else than the result of a language expansion on $\th_1$ and $\th_2$ respectively, namely, the tensor products of $\Sfm$ with those modules of theories, as in the case discussed in Section \ref{tensor}. Consequently, $\th$ is homomorphic image of both such tensor products.

\begin{theorem}\label{mergetensor}
For $i = 1, 2$, the $\Sfm$-modules $\Sfm \tensor_{{\Sfm}_i} \th_i$ and $\th_{\d_i}$ are isomorphic.
\end{theorem}
\begin{proof}
By the properties of tensor products \cite[Section 6]{russoapal} and Theorem \ref{tensembed}, $\wp E$ and $\Sfm \tensor_{{\Sfm}_i} \wp D_i$ are isomorphic to the same retract of the free module $\Sfm^{X}$, where $X$ is the set of ${\Sfm}_i$-generators of $\wp D_i$ (and of $\Sfm$-generators of $\wp E$), so the bijective map $x \in X \subseteq \wp E \mapsto 1 \tensor x \in \Sfm \tensor_{{\Sfm}_i} \wp D_i$ extends to an isomorphism $h_i: \wp E \to \Sfm \tensor_{{\Sfm}_i} \wp D_i$. On the other hand, both $\d_i$ and the nucleus $\d_i': \Sfm \tensor_{{\Sfm}_i} \wp D_i \to \Sfm \tensor_{{\Sfm}_i} \wp D_i$ whose image is $\Sfm \tensor_{{\Sfm}_i} \th_i$ are the smallest nuclei, on their respective domains, extending the ${\Sfm}_i$-module nuclei $\g_i$ and $1 \tensor \g_i$ on the isomorphic ${\Sfm}_i$-submodules $h_i[\th_i]$ and $1 \tensor \th_i$. Then the mapping $\delta_i(x) \mapsto 1 \tensor \g_i(x) = h_i(\delta_i(x))$ extends to a unique $\Sfm$-module isomorphism $h_i': \th_{\d_i} \to \Sfm \tensor_{{\Sfm}_i} \th_i$.
\end{proof}

Last, we want to stress that the above construction yields a sup-lattice of theories that contains an isomorphic copy of the sup-lattice coproduct of $\th_1$ and $\th_2$.

\begin{theorem}\label{embcoprod}
The coproduct of $\th_1$ and $\th_2$ embeds as a sup-lattice in $\th$.
\end{theorem}
\begin{proof}
Recalling that, for all $\Phi \in D_1$ and $\Psi \in D_2$, $\delta(\Phi) = \delta_1(\Phi)$ and $\delta(\Psi) = \delta_2(\Psi)$, by Theorem \ref{gammaiembedth} and \cite[Proposition 4.13]{russojlc}, the unique sup-lattice homomorphism extending the embeddings of $\th_1$ and $\th_2$ into $\th$ is 
$$e: (\g_1(\Phi),\g_2(\Psi)) \in \th_1 \amalg \th_2 \mapsto \delta(\Phi) \vee \delta(\Psi) \in \th.$$
In order to prove that it is injective, let $(\g_1(\Phi),\g_2(\Psi)), (\g_1(\Phi'),\g_2(\Psi')) \in \th_1 \amalg \th_2$ be two pairs with the same image under $e$. With an argument analogous to the one used in the proof of Lemma \ref{deltapsi}, it is not hard to see that any $\phi \in \g_1(\Phi)$ belongs to $\delta(\Phi') \vee \delta(\Psi')$ if and only if it belongs $\delta(\Phi')$ and, therefore, to $\g_1(\Phi')$ and, similarly, an the element $\psi$ of $\g_2(\Psi)$ is contained in $\delta(\Phi') \vee \delta(\Psi')$ if and only if it belongs to $\g_2(\Psi')$. This implies that $(\g_1(\Phi),\g_2(\Psi)) \leq (\g_1(\Phi'),\g_2(\Psi'))$, the converse inequality being completely analogous. The assertion follows.
\end{proof}

The relationships among the various modules of theories presented in this section can be resumed in the following commutative diagram, where $\d_i'$ denotes $\d_{\restr \th_{\d_i}}$, for $i=1,2$. We remark that the arrows of the diagram are $\Sfmm$-, $\Sfmmm$- or $\Sfm$-module morphisms, depending from the domain of the arrow in every single case, except for the dotted arrows, which are sup-lattice morphisms. For a better graphical rendering, $\delta$  do not appears as a direct arrow, it coincides with $\d_1'\d_1$ and $\d_2'\d_2$.

\begin{equation}\label{iperdiag}
\begin{tikzcd}[row sep=3em,column sep=6em]
\wp D_1 \arrow[dd, two heads,"\g_1"']  &\wp E \arrow[l, hookleftarrow, "d_1"'] \arrow[dd, two heads,"\d_1"']   &\wp D_2 \arrow[ddl,hook,"g_2",near end] \arrow[dr,two heads,"\g_2"] \arrow[l, hook, "d_2"']&\\
& &\th_{\d_2} \arrow[ull,hookleftarrow,crossing over,shift left,"g_1",near end]\arrow[dd,"\d_2'", two heads] \arrow[ul,twoheadleftarrow, crossing over,"\d_2"']  & \th_2 \arrow[ddl, hook,"e_2"'] \arrow[l,"f_2", hook] \arrow[dotted, hook, dddl, "\bot \times \id"]\\
\th_1 \arrow[r,hook,"f_1"] \arrow[drr,hook,shift right,"e_1"] \arrow[dotted, hook, shift right, ddrr, "\id \times \bot"']&\th_{\d_1} \arrow[dr,two heads,"\d_1'"]& &\\
 &&\th& \\
 &&\th_1 \amalg \th_2 \arrow[dotted, hook, "e"', u]&
\end{tikzcd}
\end{equation}

\section{Logical amalgamation}
\label{logicamalg}

Now that we have a natural and solid logical version of the coproduct of deductive systems, it is natural, as a next step, to try to handle situations in which we have two systems with a common fragment, possibly up to translations and interpretations, and we want to embed this sort of amalgam.

All the notations used in Section \ref{coproddssame} remain valid in this section. Besides that, let us add another language $\cat M$ and a deductive system $(C, \vdash_\beta)$ on $\cat M$ whose domain is, again, of the same type of $D_1$ and $D_2$, with associated nucleus $\beta$ and module of theories $\th' = \wp C_\beta$. Let us also suppose that there exist translations $\tau_i: \cat M \to \lang_i$ and structural representations $r_i: \th' \to \th_i$ via $\tau_i$, $i = 1,2$. We refer the reader to \cite{russoapal} for the definitions and results about translations and the various kinds of interpretations, including representations.

First of all, we observe that, thanks to the quantale embeddings $t_i$ induced by the translations $\tau_i$, and their compositions with the inclusion morphisms of ${\Sfm}_i$ into $\Sfm$, all of the modules and embeddings which appear in the previous section, including $e$, will now become also $\wp\Sigma_{\cat M}$-modules and $\wp\Sigma_{\cat M}$-module embeddings. By \cite[Theorem 7.1]{russoapal}, we also have two $\wp\Sigma_{\cat M}$-module morphisms $s_i: \wp C \to \wp D_i$ such that $\g_i \circ s_i = r_i \circ \beta$, $i = 1,2$.

Let $\e$ be the $\Sfm$-module nucleus on $\wp E$ associated to the consequence relation determined by the union of the axioms and rules of $\vdash_1$, $\vdash_2$, and the set of rules
$$\Theta = \left\{\frac{e_ir_i(\{\phi\})}{e_kr_k(\{\phi\})} \bigg| \phi \in C, i \neq k \in \{1,2\}\right\}.$$
If $\zeta$ is the nucleus associated to the consequence relation determined only by the rules in $\Theta$, we have clearly that $\e = \delta \vee \zeta = \delta_1 \vee \delta_2 \vee \zeta$.

In what follows, let us denote by $\e_i$ the $\Sfm$-nucleus $\delta_i \vee \zeta$ on $\wp E$, for $i = 1,2$.

In the wake of the results of the previous section, we shall prove that the consequence relation associated to $\e$ is an excellent candidate as a deductive systems which is able to amalgamate the V-formation of the initial systems without reducing the expressing power of the larger systems and languages. Moreover, we will show that the algebraic amalgamation of the given $\wp \Sigma_{\cat M}$-modules of theories embeds in this new system, thus showing, on the one hand, that the abstract algebraic construction has indeed a logical meaning and, on the other hand, that the system $(E, \vdash_\e)$ is very well behaved also from the algebraic viewpoint.

The following two propositions are not directly involved in the proofs of the main results of this section, but are useful to understand how the consequence relations we just defined work.
\begin{proposition}\label{zetadelta}
For $i \neq k \in \{1,2\}$, and for all $\Phi \cup \{\psi\} \in \wp D_k$, $\{\psi\} \subseteq \e_i(\Phi)$ if and only if one of the following two holds:
\begin{enumerate}[(i)]
\item $\psi \in \Phi$, or
\item there exist $\sigma \in \sfm$ and $\Phi' \subseteq \Phi$ such that $\Phi' = \{\sigma\} \cdot e_kr_k(\Lambda)$, $\{\psi\} = \{\sigma\} \cdot e_kr_k(\{\xi\})$, for some $\Lambda \cup \{\xi\} \in \wp C$, and $\Lambda \vdash' \xi$.
\end{enumerate}
\end{proposition}
\begin{proof}
The proof proceeds similarly to the one of Lemma \ref{deltapsi}, the right-to-left implication being, again, trivial.

First, observe that the axioms of $\vdash_{\e_i}$ coincide with those of $\vdash_{\delta_i}$ because $\zeta$ has no axioms. Since $\psi \in \e_i(\Phi)$ if and only if $\Phi \vdash_{\e_i} \psi$, by Definition \ref{consrel} we have that $\psi \in \e_i(\Phi)$ if and only if $\psi$ satisfies any of the following conditions:
\begin{itemize}
\item $\psi$ is a theorem of $\vdash_{\e_i}$,
\item $\psi \in \Phi$, or
\item $\psi$ is derivable from $\Phi$ and the axioms of $\vdash_{\delta_i}$ via inference rules of the system, i.~e., the inference rules of $\vdash_{\delta_i}$ plus those in $\Theta$.
\end{itemize}
Using $\Theta$, we have that $\psi$, being an element of $D_k$, can be a theorem of $\vdash_{\e_i}$ if and only if $\{\psi\} = \{\sigma\} \cdot e_kr_k(\{\xi\})$ for some theorem $\xi$ of $\vdash'$ and $\sigma \in \sfm$, hence this case verify the (ii) of the assertion.

Now, let us assume that $\psi$ is neither in $\Phi$ nor a theorem of $\vdash_{\e_i}$. Since $\psi \in D_k$, by Lemma \ref{deltapsi}, it can only be obtained from $\Phi$ via inference rules of $\e_i$ by a sequence of deductions starting and ending with inference rules in $\Theta$, which guarantees the existence of $\sigma \in \sfm$, $\Phi' \subseteq \Phi$ and $\Lambda \cup \{\xi\} \in \wp C$ such that $\Phi' = \{\sigma\} \cdot e_kr_k(\Lambda)$, $\{\psi\} = \{\sigma\} \cdot e_kr_k(\{\xi\})$ and $e_ir_i(\Lambda) \vdash_{\e_i} e_ir_i(\xi)$. Now, since $\vdash'$ is representable in $\vdash_i$ and the latter in $\vdash_{\delta_i}$ by Theorem \ref{gammaiembedth}, we have: $e_ir_i(\Lambda) \vdash_{\delta_i} e_ir_i(\xi)$ if and only if $r_i(\Lambda) \vdash_i r_i(\xi)$ if and only if $\Lambda \vdash' \xi$. The proof is complete.
\end{proof}

\begin{proposition}\label{epsilonik}
For $i \neq k \in \{1,2\}$, let $\Phi \in \wp D_i$ and $\psi \in D_k$. If $\{\psi\} \subseteq \e_i(\Phi)$ then there exists $\psi' \in e_kr_k(C)$ and $\sigma \in \sfm$ such that $\psi = \sigma \cdot \psi'$.
\end{proposition}
\begin{proof}
If $\psi \in \Var$, the assertion is trivially verified, so let us assume that $\psi \notin \Var$. If $\psi$ is derivable from $\Phi$, there exists $\Psi \subseteq \e_i(\Phi) \setminus \{\psi\}$ such that $\psi$ is directly derivable from $\Psi$ by means of a single application of an inference rule (one may think of $\Psi \vdash_{\e_i} \psi$ as the last step of a proof). Now, $\psi$ cannot be the consequence of an instance of a $\delta_i$-inference rule because it cannot be obtained by applying a substitution of $\sfm$ to an element of $D_i$ other than a variable, therefore there exists $\phi \in \Psi$ such that $\phi$ entails $\psi$ by means of a rule in $\Theta$. Then there exist $\xi \in C$ and $\sigma \in \sfm$ such that $\sigma \cdot e_ir_i(\xi) = \phi$ and $\sigma \cdot e_kr_k(\xi) = \psi$, and the assertion follows with $\psi' = e_kr_k(\xi)$.
\end{proof}

\begin{lemma}\label{zetadelta2}
There exist ${\Sfm}_i$-module embeddings of $\th_i$ into $\th_{\e_i}$, for $i \in \{1,2\}$.
\end{lemma}
\begin{proof}
Let us define, for $i \in \{1,2\}$, the following map:
$$l_i: \Phi \in \th_i \mapsto \e_i(\Phi) \in \th_{\e_i}.$$
For all $\Phi \in \th_i$, $\Phi = \g_i'(\Phi) = \e_i(\Phi) \cap D_i$. If $\Psi \in \th_i \setminus \{\Phi\}$, without losing generality, we can assume that there exists $\phi \in \Phi \setminus \Psi$. Then
$$\phi \in \Phi \setminus \Psi = (\e_i(\Phi) \cap D_i) \setminus (\e_i(\Psi) \cap D_i) \subseteq \e_i(\Phi) \setminus \e_i(\Psi),$$
whence $\e_i(\Phi) \neq \e_i(\Psi)$. Then the maps $l_i$ are injective. The fact that they preserve the action from ${\Sfm}_i$ and arbitrary joins is trivial.
\end{proof}

\begin{theorem}\label{epsilon}
There exist ${\Sfm}_i$-module embeddings of $\th_i$ into $\th_\e$, for $i \in \{1,2\}$.
\end{theorem}
\begin{proof}
It is immediate to verify that $\e$ can be seen also as the nucleus on $\th$ associated to the congruence generated by the relation
$$R_\e = \{(\delta_1(e_1r_1(\Psi)), \delta_2(e_2r_2(\Psi))) \mid \Psi \in \wp C\}.$$
We shall prove that $\e_i(\Phi)$ is $R_\e$-saturated for all $\Phi \in \wp D_i$ and for $i=1,2$. For the sake of readability, for $\Psi \in \wp C$, let us denote by $\Psi_i$ the set $\delta_i(e_ir_i(\Psi))$, $i = 1,2$.

Let $\Sigma \in \Sfm$, $\Psi \in \wp C$ and $\Phi \in \wp D_1$. We have:
$$\begin{array}{l}
\Sigma \cdot \Psi_1 \subseteq \e_i(\Phi) \iff \\
\zeta(\Sigma \cdot \Psi_1) \subseteq \zeta(\e_i(\Phi)) = \e_i(\Phi) \iff \\
\Sigma \cdot \zeta(\Psi_1) \subseteq \e_i(\Phi) \iff \quad \text{(because $\zeta(\Psi_1) = \zeta(\Psi_2)$)}\\
\Sigma \cdot \zeta(\Psi_2) \subseteq \e_i(\Phi) \iff \\
\zeta(\Sigma \cdot \Psi_2) \subseteq \zeta(\e_i(\Phi)) = \e_i(\Phi) \iff \\
\Sigma \cdot \Psi_2 \subseteq \e_i(\Phi).
\end{array}$$
It follows, using also Lemma \ref{zetadelta2}, that the mappings $m_i: \Phi \in \th_i \to \e_i(\Phi) = \e(\Phi) \in \th_\e$, $i=1,2$, are injective. The fact that they preserve joins and the scalar multiplication is obvious. The theorem is proved.
\end{proof}

\begin{theorem}\label{mergetensor2}
For $k \in \{1, 2\}$, $x \in \Var$, and $\Phi \in \th'$, let $\Sigma_{\Phi,k,x} \in \Sfm$ be the set of substitutions such that $\Sigma_{\Phi,k,x} \cdot \{x\} = e_kr_k(\Phi)$, i.~e., the set $e_kr_k(\Phi)/\{x\}$. Now let, for $i \neq k \in \{1, 2\}$,
$$\zeta_i: \Sfm \tensor_{{\Sfm}_i} \th_i \to \Sfm \tensor_{{\Sfm}_i} \th_i$$
be the $\Sfm$-module nucleus associated to the congruence generated by
$$Z_i = \{(1 \tensor e_ir_i(\Phi), \Sigma_{\Phi,k,x} \tensor \g_i(\{x\})) \mid \Phi \in \th', x \in \Var\}.$$
Then the $\Sfm$-modules $(\Sfm \tensor_{{\Sfm}_i} \th_i)_{\zeta_i}$ and $(\wp E)_{\e_i}$ are isomorphic.
\end{theorem}
\begin{proof}
By Theorem \ref{mergetensor}, $\Sfm \tensor_{{\Sfm}_i} \th_i$ is isomorphic to $(\wp E)_{\delta_i}$. Since $\e_i = \delta_i \vee \zeta$, it suffices to observe that $f_i' \circ \zeta = \zeta_i \circ f_i'$ in order to prove that $h_i: \zeta(\Phi) \in (\wp E)_{\e_i} \mapsto \zeta_i f_i'(\Phi) \in (\Sfm \tensor_{{\Sfm}_i} \th_i)_{\zeta_i}$ is an isomorphism.
\end{proof}

\begin{theorem}\label{embcoprod2}
The amalgamated coproduct of the $\wp \Sigma_{\cat M}$-modules $\th_1$ and $\th_2$ w.r.t. $\th'$ embeds in $\th_\e$.
\end{theorem}
\begin{proof}
By the properties of amalgamation, there exists a unique $\wp \Sigma_{\cat M}$-module morphism $e': \th_1 \amalg_{\th'} \th_2 \to \th_\e$ such that $e'((\g_1(\Phi),\g_2(\varnothing))/\theta) = \e(e_1(\Phi))$ and $e'((\g_1(\varnothing),\g_2(\Psi))/\theta) = \e(e_2(\Psi))$ for all $\Phi \in \wp D_1$ and $\Psi \in \wp D_2$, where $\theta$ is the $\wp \Sigma_{\cat M}$-module congruence on $\th_1 \amalg \th_2$ generated by the set
$$\{((\g_1(r_1(\Xi)),\g_2(\varnothing)),(\g_1(\varnothing),\g_2(r_2(\Xi)))) \mid \Xi \in \th'\}.$$
Let $\Phi, \Phi' \in \wp D_1$ and $\Psi, \Psi' \in \wp D_2$ be such that the pairs $(\g_1(\Phi),\g_2(\Psi))$ and $(\g_1(\Phi'),\g_2(\Psi'))$ are distinct $\theta$-saturated elements. By Proposition \ref{satamalg},
$$\begin{array}{l}
[\g_1(\varnothing),\g_1(\Phi)] \cap r_1[\wp C] = [\g_2(\varnothing),\g_2(\Psi)] \cap r_2[\wp C] \neq \\
\neq [\g_2(\varnothing),\g_2(\Psi')] \cap r_2[\wp C] = [\g_1(\varnothing),\g_1(\Phi')] \cap r_1[\wp C]
\end{array}$$
Without losing generality, we can assume that there exists $\phi \in (\g_1(\Phi) \setminus \g_1(\Phi')) \cap r_1[\wp C]$. Since $\e \circ e_1$ is an embedding, we get $\e e_1(\{\phi\}) \subseteq \e e_1(\Phi) \setminus \e e_1(\Phi_1')$. On the other hand, we have
$$[\g_1(\varnothing),\g_1(\Phi)] \cap r_1[\wp C] = [\g_2(\varnothing),\g_2(\Psi)] \cap r_2[\wp C], \text{ and }$$
$$[\g_1(\varnothing),\g_1(\Phi')] \cap r_1[\wp C] = [\g_2(\varnothing),\g_2(\Psi')] \cap r_2[\wp C],$$
from which we get that $\{\phi\} \subseteq \e e_2(\Psi) \setminus \e e_2(\Psi')$ with an analogous argument, and therefore
$$\begin{array}{l}
\{\phi\} \subseteq (\e e_1(\Phi) \vee \e e_2(\Psi)) \setminus (\e e_1(\Phi') \vee \e e_2(\Psi')) = \\
= e'((\g_1(\Phi),\g_2(\Psi))/\theta) \setminus e'((\g_1(\Phi'),\g_2(\Psi'))/\theta).
\end{array}$$
It follows that $e'$ is injective.
\end{proof}

The results of this section can be summarized by the following commutative diagram. All the arrows are $\wp \Sigma_{\cat M}$-module morphisms, and occasionally preserve the scalar multiplication from one or more quantales among $\Sfmm$, $\Sfmmm$, and $\Sfm$. Besides that, the morphisms $n_1$ and $n_2$ are the natural embeddings of $\th_1$ and $\th_2$, respectively, in their amalgamated coproduct w.r.t. $\th'$, and $\mu$ is the natural projection of $\wp D_1 \amalg \wp D_2$ over $\th_1 \amalg_{\th'} \th_2$.

$$\begin{tikzcd}[row sep=4em,column sep=4em]
 & \wp D_1 \arrow[rrr, bend left, hook, "d_1"] \arrow[dd, "\g_1", near end, two heads] \arrow[rr, hook, "\id \times \bot"] &  & \wp D_1 \amalg \wp D_2 \arrow[r, "d_1 \cup d_2", hook] \arrow[d, no head, "\mu"] \arrow[dd, two heads] & \wp E \arrow["\e", two heads, dd] \\
\wp C \arrow[dd, two heads, "\beta"] \arrow[rr, hook, "s_2", crossing over,  near end] \arrow[ur, hook, "s_1"] &  & \wp D_2 \arrow[urr, near end, hook, "d_2", bend right, crossing over] \arrow[ur, hook, "\bot \times \id"] & {} & \\ 
& \th_1 \arrow[rrr, bend left, "m_1"', hook] \arrow[rr, hook, "n_1", near end]  && \th_1 \amalg_{\th'} \th_2 \arrow[u, twoheadleftarrow, crossing over] \arrow["e'", r, hook] & \th_\e\\
  \th' \arrow[rr, hook, "r_2"]  \arrow[ur, hook, "r_1"] &  &  \th_2 \arrow[urr, bend right, hook, "m_2"] \arrow[ur, hook, "n_2"] \arrow[uu, crossing over, near start, twoheadleftarrow, "\g_2"]  &\\
\end{tikzcd}$$

\section*{Concluding remarks}

The results presented are readily applicable to various fields such as automated theorem provers and decision-making processes. An obvious criticism to the representation of deductive systems as quantale modules (which is in some sense shared by the author himself) is that it is not rich enough to handle first-order logics. However, it must be noted that propositional logics form the deductive skeleton of higher-order ones and therefore the results hereby presented are plainly applicable to sentential fragments of first-order systems and, more generally, to first-order logics, at least for what concerns their deductive apparatus.

Another question that is worthwhile addressing here regards the fact that we only dealt with systems of the same type, while it would be interesting to study analogous situations where the systems are syntactically different. Such situations are obviously more complex, but mainly in regard to notations and technical details, and obviously present cases which need to be treated separately. So we decided to leave those cases for future further investigations.

\end{document}